\documentclass[12pt,a4paper]{article}
\usepackage{amsmath,amssymb,amsthm,graphicx,color}
\usepackage[left=1in,right=1in,top=1in,bottom=1in]{geometry}
\usepackage{cite}
\usepackage[british]{babel}

\newtheorem{theorem}{Theorem}[section]
\newtheorem{lemma}[theorem]{Lemma}

\theoremstyle{remark}

\newtheorem{remark}{Remark}[section]
\newtheorem{remarks}[remark]{Remarks}

\allowdisplaybreaks

\numberwithin{equation}{section}

\newcommand{\ZP}{\mathbb{Z}_+}
\newcommand{\R}{\mathbb{R}}
\newcommand{\N}{\mathbb{N}}

\newcommand{\ud}{\textup{d}}

\newcommand{\rc}{{\mathrm{c}}}

\newcommand{\eps}{\varepsilon}

\newcommand{\ubar}[1]{\underline{#1\mkern-4mu}\mkern4mu }

\renewcommand{\Pr}{\mathbb{P}}
\newcommand{\Exp}{\mathbb{E}}
\newcommand{\Var}{\mathbb{V} {\rm ar}}
\newcommand{\hull}{\mathop \mathrm{hull}}
\newcommand{\as}{{\ \mathrm{a.s.}}}
\newcommand{\argmin}{\mathop{ \mathrm{arg} \min}}
\newcommand{\argmax}{\mathop{ \mathrm{arg} \max}}

\newcommand{\HH}{{\mathcal{H}}}
\newcommand{\FF}{{\mathcal{F}}}

\newcommand{\1}{{\bf 1}}
\newcommand{\be}{{\bf e}}

\newcommand{\0}{{\bf 0}}

\begin{document}

\title{Convex hulls of planar random walks with drift}
\author{Andrew R.\ Wade\footnote{Department of Mathematical Sciences,
Durham University, South Road, Durham DH1 3LE.}
\and Chang Xu\footnote{Department of Mathematics and Statistics,
University of Strathclyde, 26 Richmond Street, Glasgow G1 1XH.}}
\date{\today}
\maketitle

\begin{abstract}
Denote by $L_n$ the length of the perimeter of the convex hull of $n$ steps of a planar random walk whose
 increments have finite second moment and non-zero mean.
 Snyder and Steele
 showed that $n^{-1} L_n$ converges almost surely to a deterministic limit, and proved an upper bound on the variance $\Var [ L_n] = O(n)$.
 We show that $n^{-1} \Var [L_n]$ converges and give a simple expression for the limit,
which is non-zero for walks outside a certain degenerate class.
This answers a question of Snyder and Steele.
 Furthermore, we prove  a central limit theorem for $L_n$ in the non-degenerate case.
  \end{abstract}

\medskip


\medskip


\section{Introduction and  main results}
\label{sec:intro}

On each of $n$ unsteady steps, a drunken gardener drops a seed. Once the flowers have bloomed, what is the minimum length of fencing required to enclose the garden?

Let $Z_1, Z_2, \ldots$ be a sequence of independent, identically distributed (i.i.d.)
random vectors on $\R^2$. Write $\0$ for the origin in $\R^2$.
Define the random walk $(S_n; n \in \ZP)$ by $S_0 := \0$
and for $n \geq 1$, $S_n := \sum_{i=1}^n Z_i$.
Let
$\HH_n := \hull ( S_0, \ldots, S_n )$, the convex hull of positions of the walk up to and including the $n$th step,
and let $L_n := | \partial \HH_n |$ denote the length of the perimeter of $\HH_n$.
Assume that the increments of the random walk
have finite mean: $\Exp  \| Z_1 \|   < \infty$.

\begin{figure}[h!]
  \centering
    \includegraphics[width=0.9\textwidth]{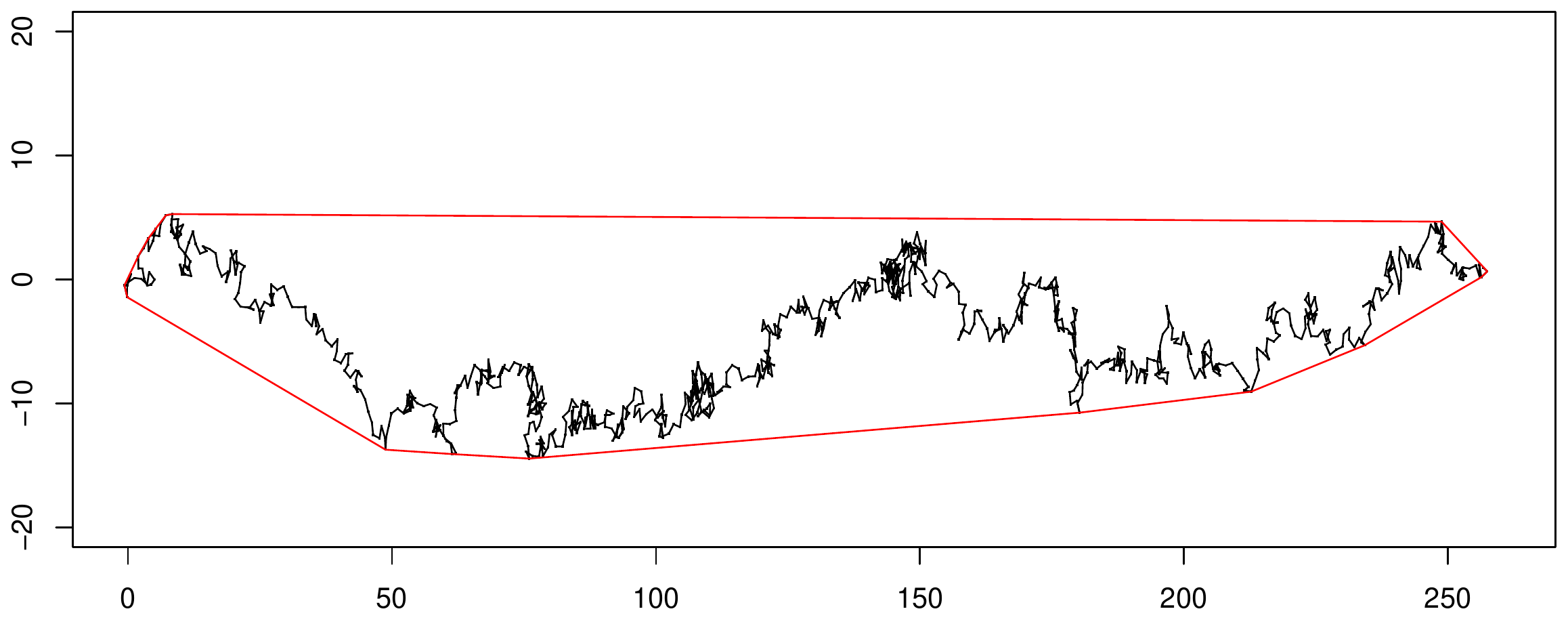}
  \caption{Example with mean drift $\Exp [ Z_1]$ of magnitude $\mu = 1/4$ and $n = 10^3$ steps.}\label{fig1}
\end{figure}

Convex hulls of random points have received much attention over the last several decades: see \cite{mcr} for an extensive survey,
 including more than 150 bibliographic references, and sources
of motivation more serious than our drunken gardener, such as modelling the `home-range' of animal populations.
An important tool in the study of random convex hulls
is provided by  a result  of Cauchy in classical convex geometry.
Spitzer and Widom \cite{sw}, using Cauchy's formula, and later Baxter \cite{baxter}, using a combinatorial argument,
showed that
\[ \Exp [ L_n ] = 2 \sum_{i=1}^n \frac{1}{i} \Exp \| S_i \| .\]
Note that  $\Exp [ L_n]$ thus scales like $n$ in the case where the one-step mean drift vector
$\Exp [ Z_1 ] \neq \0$ but like $n^{1/2}$
in the  case where $\Exp [ Z_1 ] = \0$.
The Spitzer--Widdom--Baxter result, in common with much of the literature,
is concerned with first-order properties of $L_n$:
see \cite{mcr} for a summary
of results in this direction for various random convex hulls, with a specific focus on (driftless) planar Brownian motion.

Much less is known about higher-order properties of $L_n$.
Assuming that $\Exp [ \| Z_1 \|^2 ] < \infty$,
Snyder and Steele \cite{ss} obtained an upper bound for $\Var [ L_n]$
 using Cauchy's formula together with a version of the Efron--Stein inequality.
Snyder and Steele's result  (Theorem 2.3 of  \cite{ss}) can be expressed as
\begin{equation}
\label{ssup}   n^{-1} \Var [L_n] \leq \frac{\pi^2}{2}
\left( \Exp [ \| Z_1 \|^2 ] - \| \Exp [ Z_1 ] \|^2 \right)
,  ~~~ (n \in \N := \{1,2,\ldots \} ) . \end{equation}

As far as we are aware, there are no  lower bounds  for $\Var [ L_n]$ in the literature.
According to the discussion in \cite[\S 5]{ss}, Snyder and Steele had ``no compelling reason to expect that $O(n)$ is the correct order of magnitude'' in their upper bound for $\Var [L_n]$,
and they speculated that perhaps $\Var [ L_n] = o(n)$ (maybe with a distinction between the cases of zero and non-zero drift).
Our first main result settles this question under minimal conditions, confirming
that (\ref{ssup}) is indeed of the correct order, apart from
in certain degenerate cases, while demonstrating that the constant on the right-hand side of (\ref{ssup}) is not, in general, sharp.

\begin{theorem}
\label{thm1}
Suppose that $\Exp [  \| Z_1 \|^2]   < \infty$ and
$\| \Exp [ Z_1 ] \|  \neq 0$. Then
\begin{equation}
\label{vlim}
\lim_{n \to \infty} n^{-1} \Var [ L_n ] = \frac{4 \Exp [ ( (Z_1 - \Exp [Z_1] ) \cdot \Exp [ Z_1] )^2 ]}{\| \Exp [ Z_1] \|^2} =: \sigma^2 \in [0,\infty).\end{equation}
\end{theorem}

\begin{remarks}
(i) The assumptions  $\Exp [  \| Z_1 \|^2]   < \infty$ and
$\| \Exp [ Z_1 ] \|  \neq 0$ ensure   $\sigma^2 < \infty$.

(ii) To compare the limit result (\ref{vlim}) with Snyder and Steele's upper bound (\ref{ssup}), observe that
\[ \sigma^2 = 4 \left( \frac{\Exp [ (Z_1 \cdot \Exp [ Z_1] )^2 ] - \| \Exp[ Z_1 ] \|^4 }{\| \Exp [ Z_1] \|^2 } \right)
\leq 4 \left(  \Exp[ \| Z_1 \|^2 ] - \| \Exp [ Z_1 ] \|^2 \right) .\]

(iii)
The limit $\sigma^2$ is zero if and only if $(Z_1 - \Exp [ Z_1] ) \cdot \Exp [ Z_1] =0$
with probability 1, i.e., if $Z_1 - \Exp [ Z_1]$ is always orthogonal to $\Exp[Z_1]$.
In such a degenerate case,
(\ref{vlim}) says that $\Var[L_n]=o(n)$. This is the case, for example, if $Z_1$ takes values $(1,1)$ and $(1,-1)$ each with probability $1/2$. Note that
the Snyder--Steele bound (\ref{ssup}) applied in this example says
only that $\Var [L_n] \leq (\pi^2/2) n$, which is not the correct order. Here, the two-dimensional trajectory can be viewed as a space-time trajectory of a \emph{one-dimensional} simple symmetric random walk. We conjecture that in fact $\Var [L_n] = O (\log n)$. Steele \cite{steele}
obtains variance results for the \emph{number of faces} of the convex hull
of one-dimensional simple random walk, and comments that such results for $L_n$ seem ``far out of reach'' \cite[p.\ 242]{steele}.
\end{remarks}

In the case where $\Exp [  \| Z_1 \|^2]   < \infty$ and
$\| \Exp  [ Z_1 ] \| = \mu > 0$, Snyder and Steele deduce from their bound (\ref{ssup}) a strong law of large numbers for $L_n$,
namely
$\lim_{n \to \infty} n^{-1} L_n = 2 \mu$, a.s.\ (see \cite[p.\ 1168]{ss}).
Given this and the variance asymptotics of Theorem \ref{thm1},
it is natural to ask whether there is an accompanying central limit theorem.
Our next result gives a positive answer in the non-degenerate case, again with essentially minimal assumptions.

\begin{theorem}
\label{thm2}
Suppose that $\Exp [  \| Z_1 \|^2]   < \infty$ and $\| \Exp [ Z_1 ] \|  \neq 0$. Suppose that $\sigma^2$ as defined in (\ref{vlim}) satisfies $\sigma^2>0$. Then for any $x \in \R$,
\begin{equation}
\label{clt}
 \lim_{n \to \infty} \Pr \bigg[ \frac{L_n - \Exp [ L_n ] }{\sqrt{\Var [ L_n ]} } \leq x \bigg]
=
\lim_{n \to \infty} \Pr \bigg[ \frac{L_n - \Exp [ L_n ] }{\sqrt{\sigma^2 n} } \leq x \bigg]
 = \Phi (x) ,
  \end{equation}
  where $\Phi$ is the standard normal distribution function.
\end{theorem}

 Our Theorems \ref{thm1} and \ref{thm2} will be deduced as consequences of the following result, which
 shows, perhaps surprisingly, that $L_n - \Exp [ L_n]$ can be well-approximated by a sum of i.i.d.\ random variables.

 \begin{theorem}
 \label{thm0}
 Suppose that $\Exp [  \| Z_1 \|^2]   < \infty$ and
$\| \Exp [ Z_1 ] \|  \neq 0$. Then,
as $n \to \infty$,
\[ n^{-1/2} \left| L_n - \Exp [ L_n] -   \sum_{i=1}^n \frac{ 2 (Z_i - \Exp [ Z_1] ) \cdot \Exp [ Z_1]}{\| \Exp [Z_1] \|} \right| \to 0 , \ \textrm{in} \ L^2 .\]
 \end{theorem}

The subsequent sections of the paper present the proofs of these theorems. The main ingredients, which we present
in turn, include a martingale difference representation, Cauchy's formula from convex geometry, and an analysis
of the geometry of the convex hull via extrema (the strong law of large numbers with the non-zero drift provides much of
the regularity that we need).

To finish this section we discuss some simulations. We considered a specific form of random walk with increments $Z_i - \Exp [ Z_i ] = (\cos \Theta_i , \sin \Theta_i)$,
where $\Theta_i$ was uniformly distributed on $[0,2\pi)$, corresponding to a  uniform
distribution on a unit circle centred at $\Exp [Z_i] = (\mu , 0)$, say. We took one example with  $\mu  = 0$, and
two examples with $ \mu \neq 0$ of different magnitudes. In these latter cases, the results
above take the form:
$\lim_{n \to \infty} n^{-1} \Var[ L_n ] = 4 \Exp [ \cos^2 \Theta_1 ] =  2$ (Theorem \ref{thm1})
and $(2 n)^{-1/2} (L_n - \Exp [ L_n] )$ converges in distribution to a standard normal
distribution (Theorem \ref{thm2}). The corresponding pictures in Figures \ref{sim1} and \ref{sim2}
show an agreement between the simulations and theory.

The results of this paper do not cover the case where $\| \Exp [ Z_1 ] \| =0$. The simulations in this case
suggest that, for the example we considered, $\lim_{n \to \infty} n^{-1} \Var [ L_n]$ exists (see the leftmost plot in Figure \ref{sim1}), but
 Figure \ref{sim2} does not appear to be consistent with a normal distribution as a limiting distribution. The method of the present paper provides a promising approach to the zero-drift case,
but a new idea will be needed to gain control over the geometry in that case.

\begin{figure}[h]
  \centering
    \includegraphics[width=0.32\textwidth]{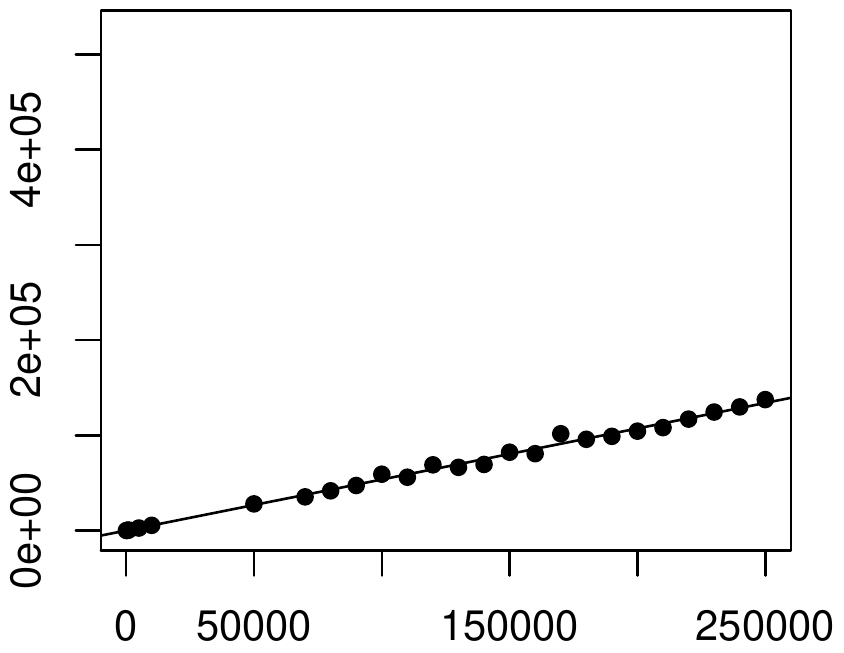}
\includegraphics[width=0.32\textwidth]{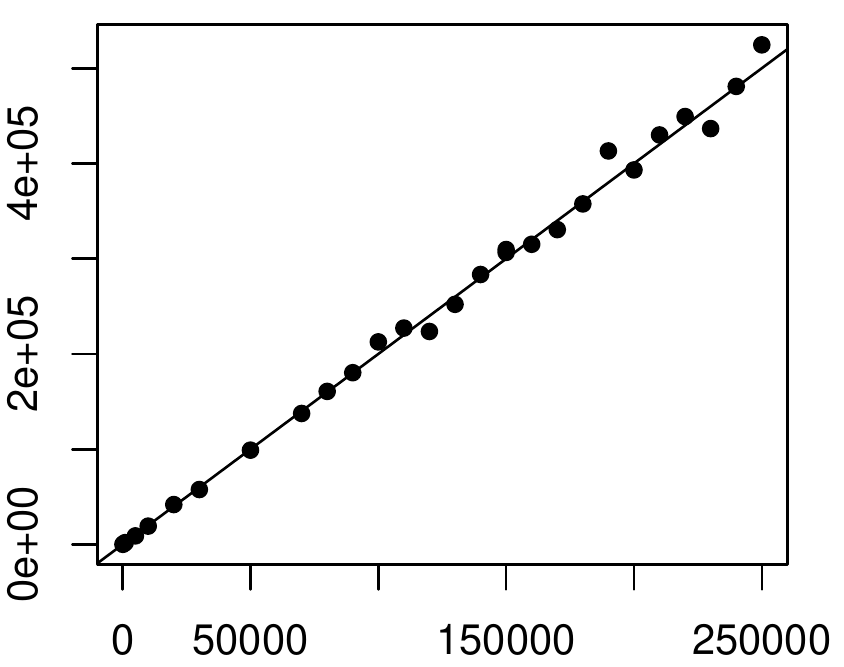}
\includegraphics[width=0.32\textwidth]{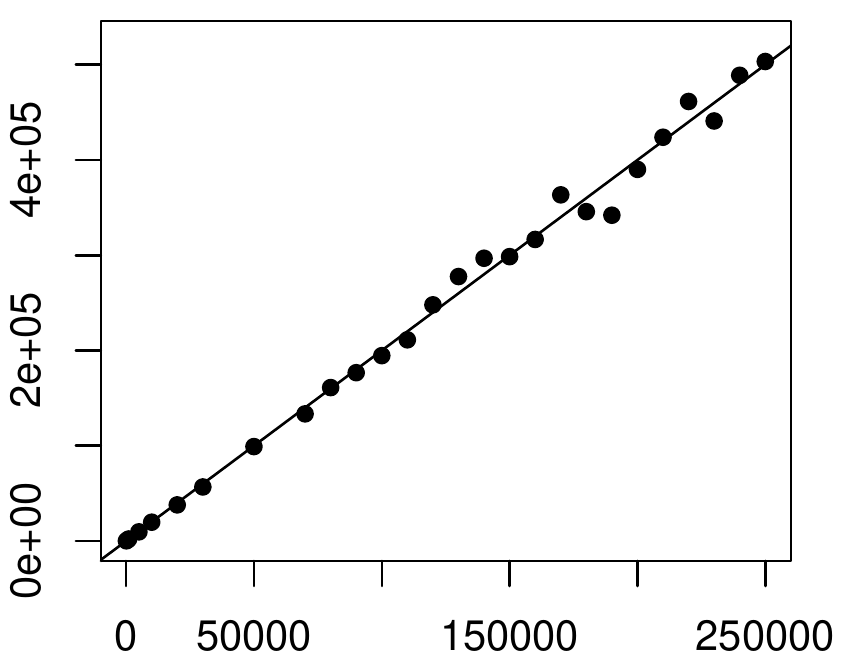}
\\
  \caption{Plots of $y = \Var[L_n]$ estimates against  $x = n$ for
about $25$ values of $n$ in the range $10^2$ to $2.5 \times 10^5$ for 3 examples
with $\mu =$ (left to right) $0$, $0.2$, $0.36$. Each point is
estimated from the sample variance of $10^3$ repeated simulations. 
Also plotted are straight lines $y = 0.536 x$ (leftmost plot)
and $y=2x$ (other two plots).
}\label{sim1}
\end{figure}

\begin{figure}[h]
  \centering
    \includegraphics[width=\textwidth]{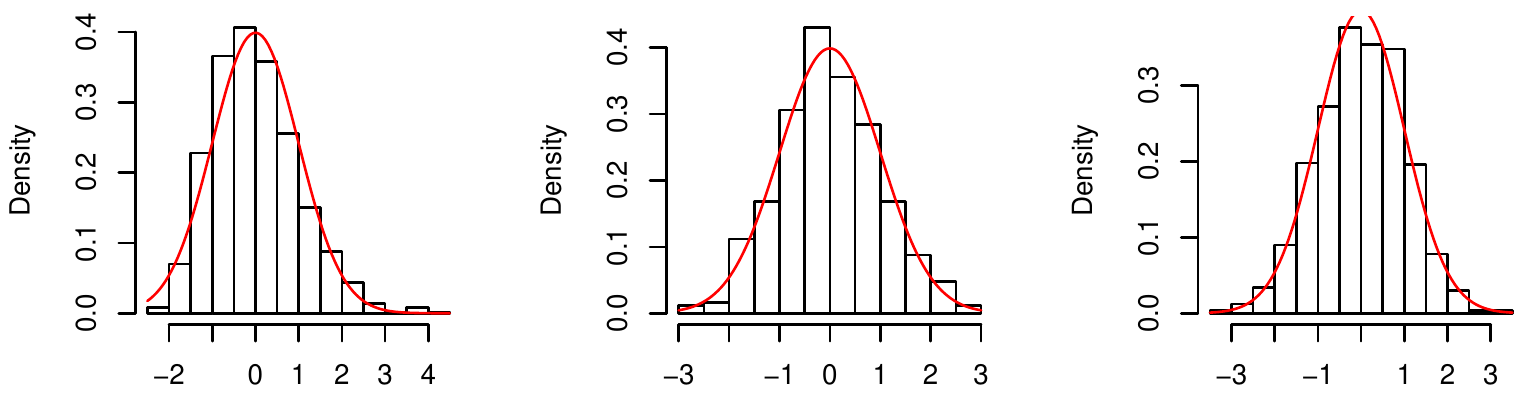}\\
  \caption{Simulated histogram estimates for the distribution of
  $\tfrac{L_n-\Exp [L_n]}{\sqrt{\Var[L_n]}}$
with  $n=5 \times 10^3$
in the three examples described in Figure \ref{sim1}. Each histogram is compiled from $10^3$ samples.}
\label{sim2}
\end{figure}

\section{Martingale difference representation}
\label{sec:martingale}

The first step in the proofs is a
 martingale difference argument, based on resampling members of the sequence
$Z_1, \ldots, Z_n$, to get an expression for $\Var [L_n]$ amenable to analysis.
Let $\FF_0$ denote the trivial $\sigma$-algebra, and for $n \in \N$ set $\FF_n := \sigma (Z_1, \ldots, Z_n)$, the $\sigma$-algebra generated
by the first $n$ steps of the random walk. Then $S_n$ is $\FF_n$-measurable, and for $n \in \N$ we can write
$L_n = \Lambda_n ( Z_1, \ldots, Z_n )$ for $\Lambda_n : \R^{2n} \to [0,\infty)$
 a measurable function.

Let $Z_1', Z_2',\ldots$ be an independent copy of the sequence $Z_1, Z_2, \ldots$.
Fix $n \in \N$. For $i \in \{1,\ldots, n\}$, we `resample' the $i$th increment, replacing $Z_i$ with $Z_i'$, as follows.
Set
\begin{equation}
\label{resample}
 S_j^{(i)} := \begin{cases} S_j & \textrm{ if } j < i \\
S_j - Z_i + Z_i' & \textrm{ if } j \geq i ;\end{cases} \end{equation}
then $(S_j^{(i)} ; 0 \leq j \leq n)$ is the random walk $(S_j ; 0 \leq j \leq n)$
but with the $i$th step independently resampled. We let $L_n^{(i)}$ denote
the perimeter length of the corresponding convex hull for this modified walk, namely
$\hull ( S_0^{(i)}, \ldots, S_n^{(i)} )$,
 i.e.,
\[ L_n^{(i)} := \Lambda_n (Z_1, \ldots, Z_{i-1}, Z'_i, Z_{i+1}, \ldots, Z_n ) .\]
For $i \in \{1,\ldots, n\}$, define
\begin{equation}
\label{dni}
 D_{n, i} :=  \Exp [ L_n - L_n^{(i)}\mid \FF_{i} ] ;\end{equation}
in words, $-D_{n,i}$ is the expected change in the perimeter length of the convex hull,
given $\FF_i$, on replacing $Z_i$ by $Z_i'$.
The point of this construction is the following result.

\begin{lemma}
\label{lem1}
Let $n \in \N$. Then (i) $L_n - \Exp [ L_n] = \sum_{i=1}^n D_{n,i}$; and (ii)
$\Var [ L_n ] = \sum_{i=1}^n \Exp [ D_{n,i}^2 ]$, whenever the latter sum is finite.
\end{lemma}
\begin{proof}
We may rewrite (\ref{dni}) as
\begin{equation}
\label{dni2}
 D_{n, i} = \Exp [ L_n \mid \FF_i ] -  \Exp [ L_n \mid \FF_{i-1} ] .\end{equation}
Indeed, since $L_n^{(i)}$ is independent of $Z_i$, $\Exp [ L_n^{(i)} \mid \FF_i ] =  \Exp [ L_n^{(i)} \mid \FF_{i-1} ] = \Exp [ L_n \mid \FF_{i-1} ]$.
Hence by (\ref{dni2}), we obtain the representation
$\sum_{i=1}^n D_{n,i} = \Exp [ L_n \mid \FF_n ] - \Exp [ L_n \mid \FF_0 ]$, giving (i).
Here $(D_{n,i} ; 1 \leq i \leq n)$ is a martingale difference sequence, since $D_{n,i}$ is
$\FF_i$-measurable and $\Exp [ D_{n,i} \mid \FF_{i-1} ] = 0$. By orthogonality of
martingale differences (see e.g.\ \cite[p.\ 218]{durrett}),
\[ \Var [ L_n ] = \Var \sum_{i=1}^n D_{n,i} = \sum_{i=1}^n \Var [ D_{n,i} ] = \sum_{i=1}^n \Exp [ D_{n,i}^2 ] ,\]
since $\Exp [ D_{n,i} ] =0$, which gives (ii).
\end{proof}

\begin{remark}
Lemma \ref{lem1}   with the conditional Jensen's inequality gives the upper bound
\[ \Var [L_n]  \leq \sum_{i=1}^n \Exp [ (L_n^{(i)} - L_n)^2 ] ,\]
which is a factor of $2$ larger than the upper bound obtained from the Efron--Stein inequality: see equation (2.3) in \cite{ss}.
\end{remark}

\section{Cauchy formula}
\label{sec:cauchy}

Let $\be_\theta = (\cos \theta, \sin \theta)$ be the unit vector in direction  $\theta \in (-\pi, \pi]$.
For $\theta \in [0,\pi]$, define
\[ M_n (\theta) := \max_{0 \leq j \leq n} ( S_j \cdot \be_\theta ) , \textrm{ and }  m_n (\theta) := \min_{0 \leq j \leq n} ( S_j \cdot \be_\theta ) .\]
Note that since $S_0 = \0$, we have $M_n (\theta) \geq 0$ and $m_n (\theta) \leq 0$, a.s.
In the present setting (see \cite{ss}, formula (2.1)), Cauchy's formula for convex sets yields
\[  L_n =  \int_0^\pi \left(  M_n (\theta) - m_n (\theta) \right) \ud \theta = \int_0^\pi R_n (\theta)  \ud \theta ,\]
where  $R_n (\theta) :=  M_n (\theta) - m_n (\theta) \geq 0$ is the {\em parametrized range function}. Similarly, when the $i$th increment
is resampled as described in Section \ref{sec:martingale},
\[ L_n^{(i)} = \int_0^\pi \left(  M^{(i)}_n (\theta) - m^{(i)}_n (\theta) \right) \ud \theta = \int_0^\pi R^{(i)}_n (\theta)  \ud \theta ,\]
where $R_n^{(i)} (\theta ) =  M^{(i)}_n (\theta) - m^{(i)}_n (\theta)$, defining
\[ M^{(i)}_n (\theta) := \max_{0 \leq j \leq n} ( S^{(i)}_j \cdot \be_\theta ) , \textrm{ and }  m^{(i)}_n (\theta) := \min_{0 \leq j \leq n} ( S^{(i)}_j \cdot \be_\theta ) .\]
Thus to study $D_{n,i} = \Exp [ L_n - L_n^{(i)} \mid \FF_i]$ we will consider
\begin{equation}
\label{cauchy}
 L_n -  L_n^{(i)} = \int_0^\pi \left(  R_n (\theta) -R^{(i)}_n (\theta)  \right) \ud \theta
 = \int_0^\pi \Delta^{(i)}_{n} (\theta) \ud \theta,\end{equation}
where $\Delta^{(i)}_{n} (\theta) :=  R_n (\theta) - R^{(i)}_n (\theta)$.
For $\theta \in [0,\pi]$, let
\[ \ubar J_{n} (\theta) := \argmin_{0 \leq j \leq n } ( S_j \cdot \be_\theta ) , \textrm{ and } \bar J_{n} (\theta) := \argmax_{0 \leq j \leq n } ( S_j \cdot \be_\theta ) ,\]
so $m_n (\theta) = S_{\ubar J_n (\theta)} \cdot \be_\theta$ and
$M_n (\theta) = S_{\bar J_n (\theta)} \cdot \be_\theta$.  Similarly, recalling (\ref{resample}), define
\[ \ubar J^{(i)}_{n} (\theta) := \argmin_{0 \leq j \leq n } ( S^{(i)}_j \cdot \be_\theta ) , \textrm{ and } \bar J^{(i)}_{n} (\theta) := \argmax_{0 \leq j \leq n } ( S^{(i)}_j \cdot \be_\theta ) .\]

We will use the following simple bound repeatedly in the arguments that follow. In fact, with a little more work one can reduce the bound on the right-hand side of (\ref{deltabound}) by a factor of 2 (cf \cite{ss}, Lemma 2.1),
but the form given here is good enough for us.

\begin{lemma}
Almost surely, for any $\theta \in [0,\pi]$ and any $i \in \{ 1,2 ,\ldots , n\}$,
\begin{equation}
\label{deltabound}
 | \Delta_n ^{(i)} (\theta ) |  \leq 2 \| Z_i\| + 2 \| Z_i ' \| . \end{equation}
\end{lemma}
 \begin{proof}
The triangle inequality implies that
\[  | \Delta_n ^{(i)} (\theta ) | \leq | M_n^{(i)} (\theta) - M_n (\theta) | + | m_n^{(i)} (\theta) - m_n (\theta) | .\]
For some $\bar J_n (\theta) \in \{0,1,\ldots, n\}$,
we have $M_n(\theta) = S_{\bar J_n (\theta)} \cdot \be_\theta$ and, by definition, $M_n^{(i)} (\theta) \geq S_{\bar J_n (\theta)}^{(i)} \cdot \be_\theta$.
If $\bar J_n (\theta) < i$, then, by (\ref{resample}),
$S_{\bar J_n (\theta)}^{(i)} = S_{\bar J_n (\theta)}$ and so $M_n^{(i)} (\theta) \geq M_n (\theta)$.
Otherwise, if  $\bar J_n (\theta) \geq  i$, then, by (\ref{resample}),
 $S_{\bar J_{n} (\theta) }^{(i)}  = S_{\bar J_{n} (\theta)} -Z_i + Z_i'$ and so
\begin{align*} M_n^{(i)} (\theta) & \geq   S_{\bar J_{n} (\theta)} \cdot \be_\theta - Z_i \cdot \be_\theta + Z_i' \cdot \be_\theta \\
& \geq M_n (\theta)- \| Z_i \| - \| Z_i ' \| . \end{align*}
Hence we conclude that, a.s., $M_n^{(i)} (\theta) \geq M_n (\theta)- \| Z_i \| - \| Z_i ' \|$.
The analogous argument in the other direction shows that $| M_n^{(i)} (\theta) - M_n (\theta) | \leq \| Z_i \| + \| Z_i ' \|$. Moreover,
a similar argument shows that the same bound holds for $| m_n^{(i)} (\theta) - m_n (\theta) |$, and (\ref{deltabound}) follows.
\end{proof}

\section{Control of extrema}
\label{sec:extrema}

For the remainder of the paper, without loss of generality, we 
suppose that $\Exp [ Z_1 ] = \mu \be_{\pi/2}$
with $\mu \in (0,\infty)$.
Observe that $( S_j \cdot \be_\theta ; 0 \leq j \leq n)$ is a one-dimensional random walk: indeed,
$S_j \cdot \be_\theta = \sum_{k=1}^j Z_k \cdot \be_\theta$. The mean drift of this one-dimensional random walk is
\begin{equation}
 \label{mutheta}
 \Exp [ Z_1 \cdot \be_\theta ] = \Exp [ Z_1] \cdot \be_\theta = \mu \sin \theta   .
\end{equation}

Note that the drift $\mu \sin \theta$ is positive if $\theta \in (0, \pi)$.
This crucial fact gives us control over the behaviour of the extrema such as $M_n  (\theta)$ and $m_n (\theta)$
that contribute to (\ref{cauchy}), and this will allow us to estimate the conditional expectation of the final term in (\ref{cauchy})
(see Lemma \ref{estimate} below).

For  $\gamma \in (0,1/2)$ and  $\delta \in (0,\pi/2)$
(two constants that will be chosen to be suitably small later in our arguments), we denote by $E_{n,i} (\delta, \gamma)$ the event that the following   occur:
\begin{itemize}
\item for all $\theta \in [\delta, \pi - \delta ]$,
$\ubar J_{n} (\theta) < \gamma  n$ and $\bar J_{n} (\theta) > (1 - \gamma) n$;
\item for all $\theta \in [\delta, \pi - \delta ]$,
$\ubar J^{(i)}_{n} (\theta) < \gamma n$ and $\bar J^{(i)}_{n} (\theta) > (1 - \gamma) n$.
\end{itemize}
We write $E^\rc_{n,i} (\delta, \gamma)$ for the complement of $E_{n,i} (\delta, \gamma)$. The  idea is that $E_{n,i} (\delta, \gamma)$ will occur with high probability, and on this event
we have good control over $\Delta^{(i)}_{n} (\theta)$.
The next result  formalizes  these assertions. For $\gamma \in (0,1/2)$, define $I_{n, \gamma} := \{1,\ldots, n\} \cap [ \gamma n , (1-\gamma) n ]$.

\begin{lemma}
\label{En}
For any $\gamma \in (0,1/2)$ and any $\delta \in (0,\pi/2)$, the following hold.
\begin{itemize}
\item[(i)]
If $i \in I_{n,\gamma}$, then, a.s., for any $\theta \in [\delta, \pi - \delta ]$,
\begin{equation}
\label{change}
  \Delta_n^{(i)} (\theta) \1 ( E_{n,i}(\delta,\gamma) )
=   ( Z_i - Z'_i ) \cdot \be_\theta  \1 ( E_{n,i}(\delta,\gamma) )  .
\end{equation}
\item[(ii)] If  $\Exp \| Z_1 \| < \infty$ and $\| \Exp [ Z_1]\| \neq 0$, then
  $\min_{1 \leq i \leq n} \Pr [ E_{n,i} (\delta, \gamma) ] \to 1$ as $n \to \infty$.
\end{itemize}
\end{lemma}
\begin{proof}
First we prove part (i).
Suppose that $i \in I_{n,\gamma}$, so $\gamma n \leq i \leq (1-\gamma) n$.
Suppose that $\theta \in [\delta, \pi-\delta ]$.
Then on $E_{n,i} (\delta,\gamma)$, we have $\ubar J_n (\theta) < i < \bar J_n (\theta)$ and
$\ubar J_n^{(i)} (\theta) < i < \bar J_n^{(i)} (\theta)$.
Then from (\ref{resample}) it follows that in fact
$\ubar J_n (\theta) =\ubar J_n^{(i)} (\theta)$
and $\bar J_n (\theta) =\bar J_n^{(i)} (\theta)$. Hence
$m_n (\theta) = m_n^{(i)} (\theta)$ and
 $M_n^{(i)} (\theta) = S^{(i)}_{\bar J_n (\theta)} \cdot \be_\theta = M_n (\theta) + (Z_i'-Z_i) \cdot \be_\theta$, by (\ref{resample}).
Equation (\ref{change}) follows.

Next we prove part (ii). Suppose that $\mu = \| \Exp [ Z_1 ] \| >0$.
Since $\Exp   \| Z_1 \| < \infty$, the strong law of large
numbers implies that $\| n^{-1} S_n - \Exp [ Z_1 ] \| \to 0$, a.s.,
as $n \to \infty$. In other words, for any $\eps_1 >0$, there
exists $N:= N (\eps_1)$ such that $\Pr [ N < \infty ] =1$
and
$ \| n^{-1} S_n - \Exp [Z_1] \| < \eps_1$ for all $n \geq N$.
In particular, for $n \geq N$, by (\ref{mutheta}),
\begin{equation}
 \label{project}
\left| n^{-1} S_n \cdot \be_\theta - \mu \sin \theta \right|
=
\left| n^{-1} S_n \cdot \be_\theta - \Exp [ Z_1] \cdot \be_\theta \right|
\leq \left\|  n^{-1} S_n - \Exp [Z_1] \right\| < \eps_1 ,\end{equation}
for all $\theta \in [0, 2\pi)$.

Take $\eps_1 < \mu \sin \delta$.
If $n \geq N$, then, by (\ref{project}),
\[ S_n \cdot \be_\theta > ( \mu \sin \theta - \eps_1 ) n \geq ( \mu \sin \delta - \eps_1 ) n,\]
provided $\theta \in [\delta, \pi-\delta]$. By choice
of $\eps_1$, the last term in the previous display is strictly
positive. Hence, for $n \geq N$, for any $\theta \in [\delta, \pi-\delta]$,
$S_n \cdot \be_\theta >0$. But, $S_0 \cdot \be_\theta =0$. So
\[ \Pr \left[  \cap_{\theta \in [\delta, \pi-\delta]}
\{ \ubar J_n (\theta) < \gamma n \} \right]
\geq \Pr [ N < \gamma n ] \to 1 ,\]
as $n \to \infty$, since $N < \infty$ a.s.

Now,
\begin{equation}
 \label{max}
 \max_{0 \leq j \leq (1-\gamma) n } S_j \cdot \be_\theta \leq
  \max \left\{ \max_{0 \leq j \leq N} S_j \cdot \be_\theta , \max_{N \leq j \leq (1-\gamma) n }
S_j \cdot \be_\theta \right\} .\end{equation}
For the final term on the right-hand side of (\ref{max}), (\ref{project}) implies that
\[ \max_{N \leq j \leq (1-\gamma) n }
S_j \cdot \be_\theta
\leq \max_{0 \leq j \leq (1-\gamma) n } ( \mu \sin \theta + \eps_1 ) j
\leq ( \mu \sin \theta + \eps_1 ) (1-\gamma ) n .\]
 On the other hand, if $n \geq N$, then (\ref{project}) implies that
$S_n \cdot \be_\theta \geq ( \mu \sin \theta - \eps_1 ) n$.
Here $\mu \sin \theta - \eps_1 \geq ( \mu \sin \theta + \eps_1 )(1-\gamma)$
if $\eps_1 < \frac{\gamma \mu \sin \theta}{2-\gamma}$.
Now we choose $\eps_1 < \frac{\gamma \mu \sin \delta}{2}$.
Then, for any $\theta \in [\delta, \pi-\delta]$, we have that,
for $n \geq N$,
\[ S_n \cdot \be_\theta >  \max_{N \leq j \leq (1-\gamma) n }
S_j \cdot \be_\theta  .\]
Hence, by (\ref{max}),
\begin{align*}
 \Pr \left[  \cap_{\theta \in [\delta, \pi-\delta]}
\{ \bar J_n (\theta) > (1-\gamma) n \} \right]
& \geq \Pr \left[
 \cap_{\theta \in [\delta, \pi-\delta]} \left\{
S_n \cdot \be_\theta >  \max_{0 \leq j \leq (1-\gamma) n } S_j \cdot \be_\theta
\right\} \right] \\
& \geq
\Pr \left[ N \leq n ,  \, \cap_{\theta \in [\delta, \pi-\delta]} \left\{
S_n \cdot \be_\theta >  \max_{0 \leq j \leq N } S_j \cdot \be_\theta
\right\}   \right] . \end{align*}
 Also, for $n \geq N$, $S_n \cdot \be_\theta > (1- \frac{\gamma}{2}) \mu n \sin \delta$,
 so we obtain
\begin{align*} \Pr \left[  \cap_{\theta \in [\delta, \pi-\delta]}
\{ \bar J_n (\theta) > (1-\gamma) n \} \right]
  \geq
\Pr \left[ N \leq n ,  \, \max_{0 \leq j \leq N} \|  S_j \| \leq \left(1- \frac{\gamma}{2}\right) \mu n \sin \delta   \right] ,
\end{align*}
using the fact that $\max_{0 \leq j \leq N}  S_j \cdot \be_\theta
\leq \max_{0 \leq j \leq N} \|  S_j \|$ for all $\theta$.

Now, as $n \to \infty$,
 $\Pr [ N > n ] \to 0$, and
\[ \Pr  \left[ \max_{0 \leq j \leq N} \| S_j \|  > \left(1- \frac{\gamma}{2}\right) \mu n \sin \delta   \right] \to 0,\]
since $N < \infty$ a.s.
So we conclude that
\[  \Pr \left[  \cap_{\theta \in [\delta, \pi-\delta]}
\{ \ubar J_n (\theta) < \gamma n , \, \bar J_n (\theta) > (1-\gamma) n \} \right] \to 1,\]
as $n \to \infty$, and the same result holds for $\ubar J_n^{(i)} (\theta)$
and $\bar J_n^{(i)} (\theta)$, uniformly in $i \in \{1,\ldots,n\}$,
 since resampling $Z_i$ does not change the distribution of the trajectory.
\end{proof}

\section{Approximation lemma}

The following result is a key component to our proof. Recall that $D_{n,i} = \Exp [ L_n - L_n^{(i)} \mid \FF_i ]$.

\begin{lemma}
\label{estimate}
Suppose that $\Exp \| Z_1 \| < \infty$, $\gamma \in (0,1/2)$, and  $\delta \in (0, \pi/2)$.
For any $i \in I_{n,\gamma}$,
\begin{align}
\label{eq2}
  \left| D_{n,i} - \frac{2 ( Z_i - \Exp [ Z_1] ) \cdot \Exp [ Z_1]}{\| \Exp [ Z_1] \|}
\right|
& \leq 
6 \delta \| Z_i \| + 6 \delta \Exp  \| Z_1\| + 3 \pi \| Z_i\| \Pr [ E_{n,i}^\rc(\delta, \gamma) \mid \FF_i ] \nonumber \\
&{} \quad {} + 3 \pi \Exp [ \| Z_i'\| \1 (  E_{n,i}^\rc(\delta, \gamma) ) \mid \FF_i  ]
, \as
\end{align}
\end{lemma}
\begin{proof}
Taking (conditional) expectations
in (\ref{cauchy}), we obtain
\begin{equation}
 \label{eq4}
 D_{n,i}
 = \int_0^\pi \Exp [ \Delta_n^{(i)} (\theta) \1 (E_{n,i} (\delta, \gamma)) \mid \FF_i ] \ud \theta + \int_0^\pi \Exp [ \Delta_n^{(i)} (\theta) \1 (E_{n,i} ^\rc (\delta,\gamma)) \mid \FF_i ] \ud \theta .\end{equation}
For the second term on the right-hand side of (\ref{eq4}), we have
\begin{align}
\label{eq5}
 \left| \int_0^\pi \Exp [ \Delta_n^{(i)} (\theta) \1 (E_{n,i} ^\rc (\delta,\gamma)) \mid \FF_i ] \ud \theta \right| &
\leq \int_0^\pi \Exp [ | \Delta_n^{(i)} (\theta) | \1 (E_{n,i}^\rc (\delta,\gamma) ) \mid \FF_i ] \ud \theta . \end{align}
 Applying the bound (\ref{deltabound}), we obtain
\begin{align}
\label{eq6}
\int_0^\pi \Exp [ | \Delta_n^{(i)} (\theta) | \1 (E_{n,i}^\rc (\delta,\gamma) ) \mid \FF_i ] \ud \theta
  \leq
2\pi \Exp [ ( \| Z_i\| + \| Z_i'\|  )  \1 (E_{n,i}^\rc (\delta,\gamma) ) \mid \FF_i ] \nonumber\\
  = 2\pi   \| Z_i \| \Pr [ E_{n,i}^\rc (\delta,\gamma)  \mid \FF_i ]
   + 2 \pi \Exp [ \| Z_i' \| \1 ( E_{n,i}^\rc (\delta,\gamma) ) \mid \FF_i ] ,\end{align}
since $Z_i$ is $\FF_i$-measurable with $\Exp   \| Z_i \|   < \infty$.

We decompose the   first integral on the right-hand side of (\ref{eq4}) as $I_1 + I_2 + I_3$,
where
\begin{align*} I_1 & := \int_0^{\delta}  \Exp [ \Delta_n^{(i)} (\theta) \1 (E_{n,i} (\delta,\gamma)) \mid \FF_i ] \ud \theta , \\
I_2 & := \int_{\delta}^{\pi-\delta}  \Exp [ \Delta_n^{(i)} (\theta) \1 (E_{n,i} (\delta,\gamma)) \mid \FF_i ] \ud \theta , \\
I_3 & := \int_{\pi-\delta}^{\pi}  \Exp [ \Delta_n^{(i)} (\theta) \1 (E_{n,i} (\delta,\gamma)) \mid \FF_i ] \ud \theta .\end{align*}

First we deal with $I_1$ and $I_3$.
We have
\begin{align*} | I_1 |    \leq \int_0^{\delta} \Exp  [ | \Delta^{(i)}_{n} (\theta) | \mid \FF_i ] \ud \theta   \leq 
2 \delta \Exp [  \| Z_i \| + \| Z_i'\|  \mid \FF_i ] ,\as,   \end{align*}
by another application of (\ref{deltabound}).
Here $\Exp [ \| Z_i \| \mid \FF_i ] = \| Z_i \|$, since $Z_i$ is $\FF_i$-measurable,
and, since $Z_i'$ is independent of $\FF_i$,
 $\Exp [ \| Z_i'\| \mid \FF_i ] = \Exp \| Z_i'\| = \Exp \| Z_1 \|$.
A similar argument applies to $I_3$, so that
\begin{equation}
\label{eq7}
|I_1+I_3| \leq 4 \delta \| Z_i \| + 4 \delta \Exp \| Z_1 \| , \as \end{equation}

We now consider $I_2$.
From (\ref{change}), since $i \in I_{n,\gamma}$,
 we have
\begin{align*} I_2 & = \int_{\delta}^{\pi - \delta} \Exp [ (Z_i - Z'_i) \cdot \be_\theta \1 (E_{n,i} (\delta,\gamma)) \mid \FF_i ] \ud \theta
  \\
& = \int_{\delta}^{\pi - \delta} \Exp [ (Z_i - Z'_i) \cdot \be_\theta \mid \FF_i ] \ud \theta
- \int_{\delta}^{\pi - \delta} \Exp [ (Z_i - Z'_i) \cdot \be_\theta \1 (E^{\rm c}_{n,i} (\delta,\gamma)) \mid \FF_i ] \ud \theta
.\end{align*}
 Here, by the triangle inequality,
\begin{align}
\left| \int_{\delta}^{\pi - \delta} \Exp [ (Z_i - Z'_i) \cdot \be_\theta \1 (E^{\rm c}_{n,i} (\delta,\gamma)) \mid \FF_i ] \ud \theta \right|
 & \leq \int_{0}^{\pi} \Exp [ ( \| Z_i\| + \| Z'_i\|) \1 (E^{\rm c}_{n,i} (\delta,\gamma)) \mid \FF_i ] \ud \theta \nonumber\\
\label{eq8}
 & {} \hskip-3cm =
\pi   \| Z_i \| \Pr [ E_{n,i}^\rc (\delta,\gamma)  \mid \FF_i ]
   +  \pi \Exp [ \| Z_i' \| \1 ( E_{n,i}^\rc (\delta,\gamma) ) \mid \FF_i ]
,  \end{align}
similarly to (\ref{eq6}).
Finally, similarly to (\ref{eq7}),
\begin{align}  &  \left| \int_{\delta}^{\pi - \delta} \Exp [ (Z_i - Z'_i) \cdot \be_\theta  \mid \FF_i ] \ud \theta
- \int_{0}^{\pi} \Exp [ (Z_i - Z'_i) \cdot \be_\theta  \mid \FF_i ] \ud \theta  \right|
  \leq 2 \delta \Exp [ \| Z_i\|  + \| Z'_i \| \mid \FF_i ] \nonumber\\
\label{eq15}
  & \qquad \qquad \qquad \qquad \qquad \qquad \qquad \qquad \qquad \qquad \qquad \qquad
   {} = 2 \delta \left( \| Z_i \| + \Exp \| Z_1 \| \right) .     \end{align}

We combine (\ref{eq4}) with (\ref{eq5}) and the bounds in (\ref{eq6}), (\ref{eq7}), (\ref{eq8}) and (\ref{eq15})
to give
\begin{align}
\label{eq29}
  \left| D_{n,i} - \int_{0}^{\pi} \Exp [ (Z_i - Z'_i) \cdot \be_\theta  \mid \FF_i ] \ud \theta \right|
& \leq 
6 \delta \| Z_i \| + 6 \delta \Exp  \| Z_1\| + 3 \pi \| Z_i\| \Pr [ E_{n,i}^\rc(\delta, \gamma) \mid \FF_i ] \nonumber \\
&{} \quad {} + 3 \pi \Exp [ \| Z_i'\| \1 (  E_{n,i}^\rc(\delta, \gamma) ) \mid \FF_i  ]
, \as
\end{align}
To complete the proof of the lemma, we compute the integral on the left-hand side
of (\ref{eq29}). 
First note that
$\Exp [ (Z_i - Z'_i) \cdot \be_\theta  \mid \FF_i ] = (  Z_i - \Exp [ Z'_i]  ) \cdot \be_\theta$, since
$Z_i$ is $\FF_i$-measurable and $Z_i'$ is independent of $\FF_i$,
so that
\[ \int_{0}^{\pi} \Exp [ (Z_i - Z'_i) \cdot \be_\theta  \mid \FF_i ] \ud \theta = \int_{0}^{\pi} (  Z_i -\Exp [ Z_i] ) \cdot \be_\theta \ud \theta.\]
To evaluate the last integral, it is convenient to introduce the notation $Z_i - \Exp [ Z_i] = R_i \be_{\Theta_i}$
where $R_i = \| Z_i - \Exp [ Z_i ] \| \geq 0$ and $\Theta_i \in [0, 2\pi )$.
Then
\begin{align*} \int_{0}^{\pi} (  Z_i -\Exp [ Z_i] ) \cdot \be_\theta \ud \theta &
= \int_0^\pi   R_i \be_{\Theta_i} \cdot \be_\theta   \ud \theta
= R_i \int_0^\pi \cos (\theta - \Theta_i) \ud \theta  \\
& =  2 R_i \sin \Theta_i = 2 R_i \be_{\Theta_i} \cdot \be_{\pi/2}.\end{align*}
Now (\ref{eq2}) follows from (\ref{eq29}), and the proof is complete.
\end{proof}

\section{Completing the proofs of the theorems}

For ease of notation, we write $Y_i := 2 \| \Exp [ Z_1] \|^{-1} ( Z_i - \Exp [ Z_1] ) \cdot \Exp [ Z_1]$, and
define
\[ W_{n,i} :=   D_{n,i} - Y_i .\]
The upper bound for $|W_{n,i}|$ in
Lemma \ref{estimate} together with Lemma \ref{En}(ii)  will enable us to prove the following result, which will be the basis of our proof of Theorem \ref{thm0}.

\begin{lemma}
\label{wni}
Suppose that $\Exp [ \| Z_1\|^2] <\infty$ and $\| \Exp [Z_1]\| \neq 0$. Then
\[ \lim_{n \to \infty} n^{-1} \sum_{i=1}^n \Exp [ W_{n,i}^2 ] = 0 . \]
\end{lemma}
\begin{proof}
Fix $\eps >0$. We take $\gamma \in (0,1/2)$ and $\delta \in (0,\pi/2)$, to be specified later. We divide the sum of interest into two parts, namely $i \in I_{n, \gamma}$ and $i \notin I_{n,\gamma}$.
 Now from (\ref{cauchy}) with (\ref{deltabound}) we have
$| L_n^{(i)} - L_n | \leq 2 \pi ( \| Z_i \| + \| Z_i' \| )$, a.s.,
so that
\[ | D_{n,i} | \leq 2 \pi \Exp [ \| Z_i \| + \| Z_i' \| \mid \FF_i ]
= 2 \pi ( \| Z_i \| + \Exp \| Z_i \| ) .\]
It then follows from  the triangle inequality that
\[ | W_{n,i} | \leq | D_{n,i} | + 2 \| Z_i - \Exp [ Z_i ] \|
\leq (2 \pi + 2) ( \| Z_i \| + \Exp \| Z_i \| ) .\]
So provided $\Exp [ \| Z_1 \|^2 ] < \infty$, we have
$\Exp [ W_{n,i}^2 ] \leq C_0$ for all $n$ and all $i$, for some
constant $C_0 < \infty$, depending only on the distribution of $Z_1$.
Hence
\[ \frac{1}{n} \sum_{i \notin I_{n,\gamma} } \Exp [ W_{n,i}^2 ]
\leq \frac{1}{n} 2 \gamma n C_0 = 2 \gamma C_0 ,\]
using the fact that there are at most $2\gamma n$ terms in the sum.
From now on, choose $\gamma >0$ small enough so that $2 \gamma C_0 < \eps$.

Now consider $i \in I_{n, \gamma}$.
For such $i$, (\ref{eq2}) shows that, for some constant $C_1 < \infty$,
\begin{equation}
\label{eq30}
 | W_{n,i} | \leq C_1 (1 + \| Z_i \| )   \delta + C_1 \| Z_i \| \Pr [ E_{n,i}^\rc (\delta,\gamma) \mid \FF_i ] ) + C_1 \Exp [ \| Z_i'\| \1 (E_{n,i}^\rc (\delta,\gamma)) \mid \FF_i ]  , \as \end{equation}
Here, for any $B_1 \in (0,\infty)$, a.s.,
\begin{align*}
 \Exp [ \| Z_i'\| \1 (E_{n,i}^\rc (\delta,\gamma)) \mid \FF_i ]
& \leq \Exp [ \| Z_i'\| \1 \{ \| Z_i' \| > B_1 \} \mid \FF_i ]
+ B_1 \Pr [ E_{n,i}^\rc (\delta,\gamma) \mid \FF_i ] \\
& = \Exp [ \| Z_i'\| \1 \{ \| Z_i' \| > B_1 \} ] + B_1 \Pr [ E_{n,i}^\rc (\delta,\gamma) \mid \FF_i ] ,
\end{align*}
since $Z_i'$ is independent of $\FF_i$.
Here, since $\Exp \| Z_i'\| = \Exp \|Z_1 \| < \infty$,
the dominated convergence
theorem implies that $\Exp [ \| Z_i'\| \1 \{ \| Z_i' \| > B_1 \} ]  \to 0$
as $B_1 \to \infty$.
So we can choose $B_1 = B_1 (\delta)$ large enough so that 
\[  \Exp [ \| Z_i'\| \1 (E_{n,i}^\rc (\delta,\gamma)) \mid \FF_i ]
\leq \delta +  B_1 \Pr [ E_{n,i}^\rc (\delta,\gamma) \mid \FF_i ], \as \]
Combining this with (\ref{eq30}) we see that there is  a constant $C_2 < \infty$ for which
\[ | W_{n,i} | \leq C_2 (1 + \| Z_i \| )  \left(
 \delta + B_1 \Pr [ E_{n,i}^\rc (\delta,\gamma) \mid \FF_i ] \right)  , \as 
\]
Hence
\begin{align*}
W_{n,i}^2 & \leq  C_2^2 (1 + \| Z_i \| )^2 \left( \delta^2 + 2B_1 \delta \Pr [ E_{n,i}^\rc (\delta,\gamma) \mid \FF_i ] +
B_1^2 \Pr [ E_{n,i}^\rc (\delta,\gamma) \mid \FF_i ]^2 \right) \\
& \leq  C_3^2 (1 + \| Z_i \| )^2 \left(  \delta + B_1^2 \Pr [ E_{n,i}^\rc (\delta,\gamma) \mid \FF_i ] \right) ,\end{align*}
for some constant $C_3 < \infty$,
using the facts that $\delta < \pi/2 < 2$ and $\Pr [ E_{n,i}^\rc (\delta,\gamma) \mid \FF_i ] \leq 1$.
Taking expectations we get
\[ \Exp [ W_{n,i}^2 ] \leq C_3^2 \delta \Exp [  (1 + \| Z_i \| )^2 ] + C_3^2 B_1^2 \Exp
\left[ (1 + \| Z_i \| )^2 \Pr [ E_{n,i}^\rc (\delta,\gamma) \mid \FF_i ] \right] .\]
Provided $\Exp [ \|Z_1 \|^2 ] < \infty$, 
there is a constant $C_4 < \infty$ such that the first term
on the right-hand side of the last display is bounded
by $C_4 \delta$. Now fix $\delta >0$ small
enough so that $C_4 \delta < \eps$; this choice also
fixes $B_1$. Then 
\begin{equation}
\label{eq32}
 \Exp [ W_{n,i}^2 ] \leq \eps  + C_3^2 B_1^2 \Exp
\left[ (1 + \| Z_i \| )^2 \Pr [ E_{n,i}^\rc (\delta,\gamma) \mid \FF_i ] \right] .\end{equation}
For the final term in (\ref{eq32}), observe that, for any $B_2 \in (0,\infty)$, a.s.,
\begin{align}
\label{eq31}
(1 + \| Z_i \|)^2 \Pr [ E_{n,i}^\rc (\delta,\gamma) \mid \FF_i ]
\leq (1+B_2)^2 \Pr [ E_{n,i}^\rc (\delta ,\gamma) \mid \FF_i ]
+ (1 + \| Z_i \|)^2 \1 \{ \| Z_i \| > B_2 \} .\end{align}
Here
$\Exp [ (1 + \| Z_i \|)^2 \1 \{ \| Z_i \| > B_2 \} ] \to 0$
as $B_2 \to \infty$, provided $\Exp [ \| Z_1 \|^2] <\infty$,
by the dominated convergence theorem.
Hence, since $\delta$ and $B_1$ are fixed,
we can choose $B_2 = B_2(\eps) \in (0,\infty)$ such that
$C_3^2 B_1^2 \Exp [ (1 + \| Z_i \|)^2 \1 \{ \| Z_i \| > B_2 \} ] < \eps$.
Then taking expectations in (\ref{eq31}) we obtain from (\ref{eq32}) that
\[ \Exp [ W_{n,i}^2 ] \leq   2 \eps + C_3^2 B_1^2  (1+B_2)^2 \Pr [ E_{n,i}^{\rm c } (\delta,\gamma) ]
  .\]

Now choose $n_0$ such that
$C_3^2 B_1^2  (1+B_2)^2 \Pr [ E_{n,i}^{\rm c } (\delta,\gamma) ] < \eps$ for all $n \geq n_0$,
which we may do by Lemma \ref{En}(ii). So for the given $\eps>0$ and  $\gamma \in (0,1/2)$,
we can choose
$n_0$ such that for all $i \in I_{n,\gamma}$ and all $n \geq n_0$,
$ \Exp [ W_{n,i}^2 ] \leq 3 \eps$. Hence
\[ \frac{1}{n} \sum_{i \in I_{n,\gamma}}  \Exp [ W_{n,i}^2 ]  \leq 3 \eps ,\]
for all $n \geq n_0$.

Combining the estimates for $i \in I_{n,\gamma}$ and $i \notin I_{n,\gamma}$, we see that
\[ \frac{1}{n} \sum_{i=1}^n  \Exp [ W_{n,i}^2 ]
\leq 2 \gamma C_0 + 3 \eps \leq 4 \eps ,\]
for all $n \geq n_0$. Since $\eps>0$ was arbitrary, the result follows.
\end{proof}

Now we can complete the proofs of our main theorems.

\begin{proof}[Proof of Theorem \ref{thm0}.]
First note that
\[ \Exp [ W_{n,i} \mid \FF_{i-1} ] = \Exp [ D_{n,i} \mid \FF_{i-1} ] - \Exp [ Y_i \mid \FF_{i-1} ]
 = 0 - \Exp [ Y_i],
\]
since $D_{n,i}$ is a martingale difference sequence and $Y_i$ is independent of $\FF_{i-1}$.
Here, by definition, $\Exp [ Y_i] =0$, and so $W_{n,i}$ is also a martingale difference
sequence. Therefore, by orthogonality,
$n^{-1}\Exp[ ( \sum_{i=1}^{n}W_{n,i})^2 ] =
n^{-1} \sum_{i=1}^{n} \Exp[ W_{n,i}^2 ]  \to 0$ as $n \to \infty$, by
Lemma \ref{wni}. In other words, $n^{-1/2} \sum_{i=1}^{n}W_{n,i} \to 0$ in $L^2$,
which implies the statement in the theorem.
\end{proof}

\begin{proof}[Proof of Theorem \ref{thm1}.]
Write
\begin{equation}
 \label{xizeta}
\xi_n = \frac{L_n - \Exp[L_n]}{\sqrt {  n} }; ~~\textrm{and} ~~
\zeta_n = \frac {1}{\sqrt{  n }} \sum_{i=1}^n Y_i,
~\textrm{where} ~ Y_i = \frac{2 (Z_i - \Exp [ Z_1 ] ) \cdot \Exp [ Z_1]}{\| \Exp [ Z_1 ] \|}.
\end{equation}
Then Theorem \ref{thm0} shows that $| \xi_n - \zeta_n | \to 0$ in $L^2$ as $n \to \infty$. Also,
with $\sigma^2$ as given by (\ref{vlim}), $\Exp [ \zeta_n^2 ] = \sigma^2$. Then a computation shows that
\[    n^{-1} \Var [ L_n ] = \Exp [ \xi_n^2 ] = \Exp [ (\xi_n - \zeta_n)^2] + \Exp [ \zeta_n^2 ] +2 \Exp [ (\xi_n -\zeta_n) \zeta_n ] .\]
Here, by the $L^2$ convergence, $\Exp [ (\xi_n - \zeta_n)^2] \to 0$ and, by the Cauchy--Schwarz inequality,
$| \Exp [ (\xi_n -\zeta_n) \zeta_n ] | \leq \left( \Exp [ (\xi_n - \zeta_n)^2] \Exp [ \zeta_n^2] \right)^{1/2} \to 0$
as well. So $\Exp [ \xi_n^2] \to \sigma^2$ as $n \to \infty$.
\end{proof}

In the proof of Theorem \ref{thm2} we will use two facts about convergence in distribution that we now
recall (see e.g.\ \cite[p.\ 73]{durrett}).
First, if sequences of random variables
$\xi_n$ and $\zeta_n$ are such that $\zeta_n \to \zeta$ in distribution for some random variable $\zeta$ and $|\xi_n -\zeta_n| \to 0$ in probability,
then $\xi_n \to \zeta$ in distribution (this is \emph{Slutsky's theorem}). Second,
if $\zeta_n \to \zeta$ in distribution and $\alpha_n \to \alpha$ in probability, then $\alpha_n \zeta_n \to \alpha \zeta$ in distribution.

\begin{proof}[Proof of Theorem \ref{thm2}.]
Suppose $\sigma^2$ as given by (\ref{vlim}) satisfies $\sigma^2 >0$.
Again use the notation for $\xi_n$ and $\zeta_n$ as given by (\ref{xizeta}).
Then, by Theorem \ref{thm0}, $| \xi_n - \zeta_n | \to 0$ in $L^2$,
and hence
in probability.

In the sum $\zeta_n$, the $Y_i$ are i.i.d.\ random variables
with mean $0$ and variance $\Exp [Y_i^2 ] = \sigma^2$. Hence the classical
central limit theorem (see e.g.\ \cite[p.\ 93]{durrett})
shows that $\zeta_n$ converges in distribution to a   normal
random variable with mean $0$ and variance $\sigma^2$.
  Slutsky's theorem then implies that $\xi_n$ has the same distributional limit. Hence, for any $x \in \R$,
\[ \lim_{n\to \infty} \Pr \left[ \frac{ \xi_n}{\sqrt{ \sigma^2 }}  \leq x \right] = \lim_{n\to \infty}  \Pr \left[ \frac{L_n - \Exp [ L_n]}{\sqrt{ \sigma^2 n}} \leq x \right] = \Phi (x) ,\]
where $\Phi$ is the standard normal distribution function.
Moreover,
\[ \Pr \left[ \frac{L_n - \Exp [ L_n]}{\sqrt{ \Var [ L_n ]}} \leq x \right]
 = \Pr \left[ \frac{\xi_n \alpha_n}{\sqrt { \sigma^2 }} \leq x \right] ,
\]
where $\alpha _n = \sqrt {\frac{\sigma^2 n}{\Var [L_n] }} \to 1$
by Theorem \ref{thm1}. Thus we verify the limit statements in (\ref{clt}).
\end{proof}

 \section*{Acknowledgements}

 Some of this work was done when the first author was at the University of Strathclyde.
 The second author is supported by a University of Strathclyde Ph.D.\ studentship, funded in part
 by the EPSRC.

\end{document}